\documentclass[11pt]{article}

\usepackage[utf8]{inputenc}
\usepackage[T1]{fontenc}
\usepackage[british]{babel}
\usepackage[a4paper, margin=2.3cm]{geometry}
\usepackage{amsmath, amsthm, amssymb, amsfonts}
\usepackage{mathtools}
\usepackage{thmtools}
\usepackage[shortlabels]{enumitem}
\usepackage[font={small, it}]{caption}
\usepackage{microtype}
\usepackage{xcolor}
\usepackage{mathabx}
\usepackage{hyperref}
\usepackage{bm}
\usepackage[normalem]{ulem}
\usepackage{tikz}
\usepackage[capitalise]{cleveref}

\newcommand{\cG}{\ensuremath{\mathcal G}}

\newcommand{\cP}{\ensuremath{\mathcal P}}

\newcommand{\cS}{\ensuremath{\mathcal S}}

\newcommand{\NN}{\ensuremath{\mathbb N}}

\renewcommand{\phi}{\varphi}
\renewcommand{\rho}{\varrho}

\let\setminus=\smallsetminus
\let\emptyset=\varnothing

\newcommand\redsout{\bgroup\markoverwith{\textcolor{red}{\rule[0.5ex]{2pt}{0.5pt}}}\ULon}

\declaretheorem[parent=section]{theorem}
\declaretheorem[sibling=theorem]{lemma}
\declaretheorem[sibling=theorem]{proposition}

\declaretheorem[sibling=theorem]{problem}

\setlist{itemsep=0.1em, topsep=0.1em, parsep=0.1em, partopsep=0.1em}

\hypersetup{
  colorlinks,
  linkcolor={red!60!black},
  citecolor={green!50!black},
  urlcolor={blue!80!black}
}

\colorlet{RoyalRed}{red!70!black}
\definecolor{RoyalBlue}{rgb}{0.25, 0.41, 0.88}
\definecolor{RoyalAzure}{rgb}{0.0, 0.22, 0.66}

\newlength{\bibitemsep}
\newlength{\bibparskip}
\let\oldthebibliography\thebibliography
\renewcommand\thebibliography[1]{%
  \oldthebibliography{#1}%
  \setlength{\parskip}{\bibitemsep}%
  \setlength{\itemsep}{\bibparskip}%
}

\def\cleq{\preccurlyeq}

\title{Smaller universal posets}

\usepackage[noblocks]{authblk} 

\newcommand\emailfont{\sffamily}
\newcommand*{\email}[1]{\href{mailto:#1}{\emailfont#1} } 

\author[1]{Paul Bastide\thanks{Research supported by ERC Advanced Grant 883810.}}
\author[2]{Carla Groenland\thanks{Research supported by the Dutch Research Council (NWO, VI.Veni.232.073).}}
\author[3]{Rajko Nenadov\thanks{Research supported by the New Zealand Marsden Fund 23-UOA-117.}}

\affil[1]{%
Mathematical Institute, University of Oxford, UK. Email: \email{paul.bastide@ens-rennes.fr}}

\affil[2]{%
Delft Institute of Applied Mathematics, TU Delft, the Netherlands. Email: 
\email{c.e.groenland@tudelft.nl}
}

\affil[3]{%
School of Mathematics and Statistics, University of Canterbury, New Zealand. Email: \email{rajko.nenadov@canterbury.ac.nz}
}
\date{\today}

\begin{document}

\maketitle

\begin{abstract}
    We show that there is a constant $C>0$ such that for each integer $n\geq 1$, there is a poset on at most $2^{2n/3+C\sqrt{n}}$ elements that contains each $n$-element poset as an (induced) subposet. 
\end{abstract}

\section{Introduction}
A graph $U$ is said to be \textit{universal} for a family of graphs $\cG$ if for any graph $G \in \cG$, $U$ contains\footnote{In this work \emph{contains} will always refers to the \emph{induced} containment relation.} $G$. Moreover if $U \in \cG$, then we call $U$ \textit{faithful}. This concept emerged in the 1950s~\cite{fraisse1953theory,johnston1956universal,rado1964universal,moon65}, and has been extensively studied since then~\cite{komjath1984universal,diestel1985universal,komjath1988some,furedi1997nonexistence,furedi1997existence,komjath1999some,cherlin1999universal,cherlin2001forbidden,cherlin2007universalforbidden,cherlin2007universal,cherlin2016universal,huynh2021universality,bastide2025faithful}. A well-known example of this notion is the \emph{Rado graph}. Ackermann~\cite{ackermann1937widerspruchsfreiheit}, Erd\H{o}s and Rényi~\cite{erdos1963asymmetric} and Rado~\cite{rado1964universal} independently proved that the Rado graph is universal for the class of all countable graphs. The Rado graph has, since then, proven to be useful in solving numerous other questions in combinatorics (see the following surveys~\cite{cameron1997random,cameron2001random}). Universal structures are also of interest to the computer science community because of their close connection to efficient data structures~\cite{fraigniaud2010optimal,alstrup2015adjacency,abrahamsen.alstrup.ea:near-optimal,MunroNicholson,alstrup2017optimal,ACKRRS,EJM23}. 

Similarly, a poset $U$ is said to be universal for a family of posets $\cP$ if $U$ contains all posets $P\in \cP$.
This notion is arguably even older than universality for graphs. Indeed, the existence of a countable universal poset containing all countable posets has been proven multiple times~\cite{fraisse1953theory,jonsson1956universal,johnston1956universal} and, as written in~\cite{hubivcka2005universal}, in the context of category theory, ``\emph{motivated the whole research area}''.

For finite posets, similar questions have been asked. This paper studies the following natural problem which was highlighted in the conclusion of~\cite{BonamyEsperetGroenlandScott21}, and originates from a question posed by Hamkins~\cite{MOquestion}.
\begin{problem}
\label{pb:uni_min_poset}
    What is the minimum order of a poset that contains every $n$-element poset? 
\end{problem}

Posets have also been studied through the graph they induce, called comparability graphs\footnote{See \cref{sec:prelim} for a formal definition.}. 
For any pair of posets $P$ and $U$ where $P$ is contained in $U$, the comparability graph of $P$ will be an induced subgraph of the comparability graph of $U$. This means that the answer to \cref{pb:uni_min_poset} is at least the answer to the following problem.
\begin{problem}
    \label{pb:uni_min_comparability}
    What is the minimum order of a comparability graph which contains every $n$-vertex comparability graph as an induced subgraph?
\end{problem}

The Boolean lattice $Q_n=(2^{[n]},\subseteq)$ on $2^n$ elements, consisting of the subsets of $\{1,\dots,n\}$ as elements and set inclusion $\subseteq$ as partial order, is well-known to contain all posets on $n$ elements (see \cref{prop:hypercube_univ}), and therefore gives an upper bound of $2^n$ to both \cref{pb:uni_min_poset} and \cref{pb:uni_min_comparability}. Hamkins~\cite{MOquestion} asked whether any substantial improvement can be made on this upper bound. Regarding lower bounds, the size of a universal poset $U$ is at least $2^{(1+o(1)) n/4}$, which can be derived from the following simple information-theoretic argument. Consider two antichains of size $n/2$, one containing minimal elements and the other maximal elements, and any subset of the $(n/2)^2$ possible order relations from a minimum to a maximum between the two antichains. From this, we deduce that there are at least $2^{n^2/4}/n!$ posets on $n$ elements.  A universal poset $U$ must then satisfy $\binom{|U|}{n} \geq 2^{(1+o(1))n^2/4}$, which implies the lower bound $|U|\geq 2^{(1+o(1))n/4}$. We remark that the lower bound on the number of posets above is asymptotically correct, since Kleitman and Rothschild~\cite{KleitmanRothschild75} proved that there are $2^{n^2/4 +3n/2+O(\log n)}$ posets on $n$ elements.

The lower bound $2^{(1+o(1))n/4}$ and upper bound $2^{(1+o(1))n}$ were the best bounds known for \cref{pb:uni_min_poset} and \cref{pb:uni_min_comparability} before our work.  
We improve on the upper bound by 
showing that an exponentially small fraction of the Boolean lattice already embeds all posets on $n$ elements. 
\begin{theorem}
    \label{thm:uni_poset_2/3_hypercube}
    There is a constant $C>0$ such that for any $n \in \NN$, there is a subposet $P_n$ of $(2^{[n]},\subseteq)$ on at most $2^{\frac23 n + C\sqrt{n}}$ elements that contains all $n$-element posets. 
\end{theorem}
In particular, there is an induced subgraph $U_n$ of the hypercube comparability graph which contains all comparability graphs on $n$ vertices as an induced subgraph with $|V(U_n)| \leq 2^{\frac23 n + C\sqrt{n}}$.

\paragraph{Further related work}
A relaxation of \cref{pb:uni_min_comparability}, where the universal object representing all posets does not need to be a poset but can be any kind of data structure, has been the focus of considerable research. Munro and Nicholson~\cite{MunroNicholson} designed a succinct data structure using only \(n^2/4 + o(n)\) bits for this problem, which supports precedence queries in constant time. 
Dul\k{e}ba, Gawrychowski and Janczewski~\cite{DulebaGawrychowskiJanczewski} provided a data structure, called \emph{comparability labelling scheme}, which assigns short labels to each element of the posets such that, given the labels of two elements, and no other information regarding the poset, it is possible to decide whether these two elements are comparable. In~\cite{DulebaGawrychowskiJanczewski}, the authors showed that $(1+o(1))n/3$ bits of label per element is sufficient. This was improved to $(1+o(1))n/4$ by Bonamy, Esperet, Groenland and Scott~\cite{BonamyEsperetGroenlandScott21}, which is asymptotically optimal. This last results imply the existence of a universal graph on $2^{(1+o(1))n/4}$ vertices that contains all comparability graphs as induced subgraphs. However, as also mentioned in~\cite{BonamyEsperetGroenlandScott21}, these universal graphs do not need to be (and will not be) comparability graphs themselves so do not yield a universal poset.

The minimum integer $k$ such that a poset $P$ can be embedded in $Q_k$ is an extensively studied poset parameter, introduced by Novak in 1969 \cite{novak1969well}, and known as the $2$-dimension of $P$. For more information on this parameter we refer the reader to Trotter's book \cite{trotter2002combinatorics}.

The strength of the hypercube as a universal structure has also been studied in terms of the minor relation: Benjamini, Kalifa and Tzalik~\cite{benjamini2025hypercubeminoruniversality} showed that the $d$-dimensional hypercube $\{0,1\}^d$ 
contains all graphs with $O(2^d/d)$ edges and no isolated vertices as a minor, and this bound was shown to be best possible by Hogan, Michel, Scott, Tamitegama, Tan and Tsarev~\cite{hogan2025tightboundshypercubeminoruniversality}.

In~\cite{fraigniaud2010optimal}, Fraigniaud and Korman studied data structures called \emph{ancestry labeling schemes}, which aim to assign a label to every vertex in a tree such that, for any two vertices $u$ and $v$, it is possible to determine whether $u$ is an ancestor of $v$ by inspecting only the labels of $u$ and $v$. They used this labeling scheme, to design small universal posets for posets of bounded \emph{tree-dimension} (see~\cite{fraigniaud2010optimal} for formal definitions).

In the context of finite graphs, Gol'dberg and Livshits~\cite{GL68} showed the existence of a tree on $2^{O(\log^2 n)}$ vertices which is universal for the class of $n$-vertex trees, which was later proven to be the optimal order of magnitude~\cite{CGC81}. Recently,  Bergold, Ir\v{s}i\v{c}, Lauff, Orthaber, Scheucher and Wesolek~\cite{BILOSW24} proved that both the class of planar graphs and the class of outerplanar graphs do not admit a faitful universal element of subexponential size. 
Their result was then strengthened by Bastide, Esperet, Groenland, Hilaire, Rambaud, and Wesolek~\cite{bastide2025faithful}, who proved that even when the requirement for the universal graph to be planar is relaxed to forbidding any $K_t$-minor for some constant $t \geq 5$, an exponential lower bound still holds. The authors in~\cite{bastide2025faithful} also studied bounds on faithful (and near-faithful) universal graphs for multiple classes of graphs (bounded degree, treewidth, pathwidth, etcetera).  

\section{Preliminaries}
\label{sec:prelim}

\paragraph{Partially ordered sets} Given a set of elements $P$ and a binary relation $\cleq$ (a subset of $P\times P$), we say that $(P,\cleq)$ is a \emph{partially ordered set} (or a \emph{poset}) if $\cleq$ satisfies the following three properties.

\begin{itemize}
    \item Reflexivity: For all $a \in P$, $a \cleq a$;
    \item Anti-symmetry: For all $a,b \in P$, if $a \cleq b$ and $b \cleq a$ then $a=b$;
    \item Transitivity: For all $a,b,c \in P$, $a \cleq b$ and $b \cleq c$ then $a \cleq c$.
\end{itemize}

For a partial order $(P,\cleq)$ we define a \emph{strict} partial order as $(P,\cleq \setminus\{(e,e) \mid e \in P\})$. Note that there is a direct correspondence between a partial order and the strict partial order it induces and vice versa. We refer to $\prec$ as the strict partial order of $\cleq$. 
Moreover, when the context is clear, we refer to $P$ as the poset instead of $(P,\cleq)$.

The poset induced by the total order on $n$ elements $([n],\{(i,j) \mid i \leq j\})$ is called the $n$-element \textit{chain} and the poset induced by $n$ pairwise incomparable elements $([n],\left\{(i,i) \mid i \in [n]\right\})$ the $n$-element \textit{antichain}. We say that a poset \((U,\leq)\) \emph{contains} another poset \((P,\cleq)\) as an induced subposet if there exists \(P' \subseteq U\) such that \((P',\leq)\) is isomorphic to \((P,\cleq)\). 
A family of subsets $\{C_1,\ldots,C_k\}$, $C_i \subseteq P$, is called a \textit{chain decomposition} of size $k$ of a poset $P$  if $C_i$ is isomorphic to a chain for every $i \in [k]$ and $(C_i)_{i \in [k]}$ forms a partition of $P$. The following theorem, known as Dilworth's theorem, is a fundamental result on posets.

\begin{theorem}[Dilworth~\cite{dilworth1950decomposition}]
    \label{dilworth}
    Let $P$ be a finite poset. Then the size of a largest antichain contained in $P$ is equal to the minimum size of a chain decomposition of $P$.
\end{theorem}

\paragraph{Boolean lattice as universal poset}
A poset $U$ is called \emph{universal} for a class of posets $\cP$ if it contains all posets $P \in \cP$. We state the following well-known fact about the Boolean lattice  below and include a short proof for completeness. 

\begin{proposition}
    \label{prop:hypercube_univ}
    The poset $(2^{[n]},\subseteq)$ is universal for the class of all $n$-element posets.
\end{proposition}

\begin{proof}
    Consider an $n$-element poset $(P,\cleq_P)$ and let $P=\{v_1,\dots,v_n\}$. Given $v_j \in P$, we define $S_j = \{ i \in [n] \mid v_i \cleq_P v_j \}$. We claim that the subposet of $(2^{[n]},\subseteq)$ defined by $P' = (\{S_i \mid i \in [n]\}, \subseteq)$ is isomorphic to $P$.
    Let $v_i,v_j\in P$. If $v_i \cleq_P v_j$ then for all $v_k \in V$, if $v_k \cleq_P v_i$ then also $v_k \cleq_P v_j$ by transitivity, so $S_i \subseteq S_j$. On the other hand, if $v_i\not\cleq v_j$, then $i \in S_i \setminus S_j$  and so $S_i\not\subseteq S_j$.
\end{proof}

\paragraph{Comparability graphs} Given a poset $(P,\cleq)$ we can define the graph $G = (P,E(G))$ with $E(G) = \left\{\{a,b\} \in \binom{P}{2} \mid a \cleq b \right\}$. $G$ is called the \emph{comparability graph} of $P$. We also say that a graph $H$ is a comparability graph, if there exits a poset $(P,\cleq)$ for which $H$ is the comparability graph.

\paragraph{Partition number} We will encounter the partition function while constructing universal posets. Let us define $p(n)$ as the number of ways we can \emph{partition} an integer into smaller integers. Formally, $p(n)$ is the number of integer sequences $\lambda_1
\geq \lambda_2\geq  \ldots\geq  \lambda_k \geq 1$ such that $\sum_{i=1}^{k} \lambda_i = n$. We refer the reader to~\cite{andrews1998theory} for more details, including a proof of the following theorem.

\begin{theorem}
    \label{thm:partition}
Let $n$ be an integer and let $p(n)$ denote the number of partitions of $n$. As $n\to \infty$, 
\[
p(n) \sim \frac{1}{4\sqrt{3}n} \exp\left(\pi \sqrt{\frac{2n}{3}}\right).
\]
\end{theorem}

\section{Proof of \cref{thm:uni_poset_2/3_hypercube}}
The proof of Theorem~\ref{thm:uni_poset_2/3_hypercube} takes the union of two constructions, one made for posets with a large antichain and the other for posets without a large antichain.  

We first describe our construction for posets without a large antichain.
\begin{lemma}
\label{lem:hypercube_chain}
    Let $1\leq a\leq n$ be integers and let $p(n)$ denote the number of partitions of $n$.
    There exists a set family $\mathcal{S}\subseteq 2^{[n]}$ with $|\mathcal{S}|\leq p(n)(n/a+1)^a$ such that $(\mathcal{S},\subseteq)$ contains all posets on $n$ elements that do not have an antichain of size $a+1$. 
\end{lemma}
\begin{proof}
We create a set family
$\mathcal{S}_n(c_1,\dots,c_a)$ for each integer sequence $n\geq c_1\geq \dots\geq c_a\geq 1$ with $\sum_{i=1}^ac_i=n$ and then take $\mathcal{S}$ to be the union of all these set families. The integer sequences will represent the lengths of each chain in a chain decomposition.

Let $n\geq c_1\geq \dots\geq c_a\geq 1$ be given with $\sum_{i=1}^ac_i=n$.
Let $\ell_0=0$ and for $i\in [a]$, let 
\[
\ell_i = \sum_{j=1}^i c_i \text{ and }C_i=\{\ell_{i-1}+1,\ell_{i-1}+2,\dots,\ell_i\}.
\]
Let $\mathcal{S}_n(c_1,\dots,c_a)$ denote the family of subsets $S\subseteq[n]$ that for each $i\in[a]$ satisfy
\[
S\cap C_i \in \left\{\emptyset,\left\{\ell_{i-1}+1\right\},\{\ell_{i-1}+1,\ell_{i-1}+2\},\dots,C_i\right\}.
\]
Since $\{C_i:i\in [a]\}$ forms a partition of $[n]$ and there are $c_i+1$ options for $S\cap C_i$, we find that $|\mathcal{S}_n(c_1,\dots,c_a)|=\prod_{i=1}^a(c_i+1)$. Since the function $f : (z_1\ldots,z_a) \mapsto \prod_{k \in [a]} (z_i +1)$ is convex and $\sum_{i=1}^a c_i = n$, we find that  $|\mathcal{S}_n(c_1,\dots,c_a)|\leq\prod_{i=1}^a(c_i+1)\leq  (\frac{n}a+1)^a$ since $f$ is maximised when $z_k = n/a$ for all $k\in [a]$. The number of choices for integer sequences $n\geq c_1\geq \dots\geq c_a\geq 1$ is at most $p(n)$ so $|\mathcal{S}|\leq p(n)(n/a+1)^a$.

Next we show that the Boolean lattice induced on $\mathcal{S}$ is universal for the posets on $n$ elements without an antichain of size $a+1$. Let $P$ be one such poset. By Dilworth's theorem (\cref{dilworth}), we can find a chain decomposition of $P$ using $a$ non-empty chains that covers all elements of $P$. We number the chains such that $n\geq c_1\geq \dots\geq c_a\geq 1$ are the lengths of the chains. Then $\sum_{i=1}^ac_i=n$.

We use $\cleq_P$ for the partial order relation on $P$. For $i\in [a]$, with $\ell_i=\sum_{j=1}^ic_i$ as defined above, we rename the elements of the $i^\text{th}$ chain of $P$ (of length $c_i$) by $x_{\ell_{i-1}+1},x_{\ell_{i-1}+2},\dots,x_{\ell_{i}}$, where the numbering respects the chain in the sense that
\[
x_{\ell_{i-1}+1}\cleq_Px_{\ell_{i-1}+2}\cleq_P\dots \cleq_P x_{\ell_{i}}.
\]
Since a chain decomposition is a partition, each element of $P$ is renamed uniquely. 
Next, we associate to an element $x_j$ the subset
\[
S_j=\{i\in [n]:x_i \cleq x_j\}.
\]
As shown in \cref{prop:hypercube_univ}, $S_i\subseteq S_j$ if and only if $x_i\cleq x_j$, therefore $(\{S_j\}_{j \in [n]},\subseteq)$ is isomorphic to $P$. What remains to show is that $S_j\in \mathcal{S}$ for all $j\in [n]$. 
Let $i\in [a]$. If there is no $m\in \{\ell_{i-1}+1,\ell_{i-1}+2,\dots,\ell_i\}=C_i$ with $x_m\cleq x_j$, then $S_j$ intersects $C_i$ in the empty set which is allowed. If there is such an $m$, then let $m$ be the largest index among $C_i$ such that $x_m \cleq_P x_j$. We claim that $S_j\cap C_i=\{\ell_{i-1}+1,\dots,m\}$. Indeed, for each $m'\in \{\ell_{i-1}+1,\dots,m\}$ , we find that
\[
x_{m'}\cleq_P x_m\cleq_P x_j,
\]
so $m'\in S_j$. Similarly, by maximality of $m$, any $m' \in \{m+1, \ldots, \ell_{i}\}$ satisfies $m' \notin S_j \cap C_i$, This shows that $S_j\in \mathcal{S}_n(c_1,\dots,c_a)\subseteq \mathcal{S}$ for all $j\in [n]$ and so we found an embedding of $P$ into the subposet of the Boolean lattice induced on $\mathcal{S}$.
\end{proof}
The next construction is for posets with a large antichain.
\begin{lemma}
\label{lem:hypercube_antichain}
    Let $n\geq a\geq 2$ be integers. Let $\mathcal{A}_n(a)$ denote the family of all $n$-element posets $P$ which contain an antichain of size $a$. Let $\ell$ be the smallest integer for which $\binom{\ell}{\lfloor\ell/2\rfloor}\geq a$. 
    Then the Boolean lattice $(2^{[n-a+\ell]},\subseteq)$ is a universal poset for $\mathcal{A}_n(a)$. 
\end{lemma}
Since $\binom{2x}{x}=\Theta(2^{2x}/\sqrt{x})$, as $a\to \infty$ the value $\ell$ from the lemma above will be $\ell=(1+o(1))\log_2a$.
\begin{proof}[Proof of Lemma~\ref{lem:hypercube_antichain}]
In this proof, we use the notation $[a,b]$ for the set of integers $x$ with $a\leq x\leq b$.

    Let $(P,\cleq_P) \in \mathcal{A}_n(a)$ and let $A \subseteq P$ be an antichain of size $a$. Let $b$ denote the number of elements $x\in P$ with $x\cleq_P y$ for some $y\in A$. We number all such elements of $P$ ``below'' $A$ by $x_1,\dots,x_b$. We assign each element in the antichain $A$ a unique (and arbitrary) subset $S\subseteq [b+1,b+\ell]$ of size $\lfloor \ell/2 \rfloor$ and name the element itself $y_S$. (Note that this is possible by our choice of $\ell$.) We number the remaining $n-a-b$ elements by $z_{b+\ell+1},\dots,z_{n-a+\ell}$.  

    The elements $x_i$ can be seen as ``below'' the antichain and the elements $z_j$ as ``above'' the antichain in the following sense. Firstly, $y_S\not\cleq_P x_i$ for each $i$ and $S$ (using that $A$ is an antichain, we cannot have $y_{S}\cleq_P x_i\cleq_P y_{S'}$). Secondly, $z_j\not\cleq_P y_S$ and $z_j\not\cleq_P x_i$ as $z_j$ is not ``below'' an element in the antichain by definition and $x_i$ is. 
    Given $v \in P$, we define $f(v)\in 2^{[n-a+\ell]}$ depending on whether $v$ is ``below'' $A$, in $A$ or ``above'' $A$.

    For $i\in [1,b]$, 
    \[
    f(x_i)=\{j\in [b]:x_j\preceq x_i\}.
    \]
For $S\subseteq [b+1,b+\ell]$, if $y_S$ is defined, \[
f(y_S)=\{i\in [b]:x_i\cleq_P y_S\}\cup S\cup  \{j\in [b+\ell+1,n-a+\ell]:y_S\not \cleq_P z_j\}.
\]
For $j\in [b+\ell+1,n-a+\ell]$,  \[
f(z_j)=\{i\in [b]:x_i\cleq_P z_j\} \cup [b+1,b+\ell]
\cup 
\{i\in [b+\ell+1,n-a+\ell]:z_j \not \cleq_P z_i\}. 
\]
What remains to show is that for all elements $u,v\in P$, we have
\begin{equation}
    \label{eq:poset_contains}
u\cleq_P v \iff f(u)\subseteq f(v).
\end{equation}

\paragraph{Case 1: $u = x_i$.} Note that, similar to the proof of  Proposition~\ref{prop:hypercube_univ}, for all $i\in [b]$ and $v\in P$ we have $x_i\cleq_P v$ if and only if $i\in f(v)$. This implies that
\[
x_i\cleq_P v \iff f(x_i)\subseteq f(v)
\]
for the same reason as in that proof.
Indeed, if  $x_i\not\cleq_Pv$, then $i\in f(x_i)\setminus f(v)$. On the other hand, if $x_i\cleq_P v$, then for all $j\in f(x_i)$, $x_j\cleq_P x_i$ implies that $x_j\cleq_Pv$  and so $j\in f(v)$. 

\paragraph{Case 2: $u = y_S$.} 
We consider the case $u=y_S\in A$ by case analyis on $v$.
\begin{itemize}
    \item If $v=x_i$ for some $i\in [1,b]$, then $y_S\not\cleq_P x_i$ by the definition of ``below''. We also find that $S\subseteq f(y_S)\setminus f(x_i)$ and so $f(y_S)\not\subseteq f(x_i)$, as desired.
    \item If $v=y_{S'}\in A$ with $u\neq v$, then $u,v$ are incomparable in $P$. As $S$ and $S'$ have the same size and $S \neq S'$, we conclude $S \not \subseteq S'$. Therefore, $f(u) \not \subseteq f(v)$.
    
    \item If $v=z_j$ for some $j\in [b+\ell+1, n - a + \ell]$ and $y_S\not\cleq z_j$, we find $j\in f(y_S)\setminus f(z_j)$, where $j \not \in f(z_j)$ since $z_j \cleq_P z_j$. So $f(y_s)\not \subseteq f(z_j)$. 
    
    On the other hand, if $y_S\cleq_P z_j$, then for all $i\in [b]$ by definition of $f$ and transitivity \[
    i\in f(y_S)\implies 
    x_i\cleq_P y_S\implies x_i\cleq_P z_j \implies i\in f(z_j).
    \]
    For the same reasons for all $i\in [b+\ell+1,n-a+\ell]$,
    \[
    i\in f(y_S)\implies y_S\not\cleq_P z_i\implies z_j \not\cleq_P z_i\implies i\in f(z_j).
    \]
    By definition, $[b+1,b+\ell]\subseteq f(z_j)$. This shows that $f(y_S)\subseteq f(z_j)$ if $y_S\cleq_P z_j$.
\end{itemize}

\paragraph{Case 3: $u = z_j$.} Finally, we consider the case where $u=z_j$ for $j\in [b+\ell+1,n-a+\ell]$. Now $u\not\cleq v$ for all $v$ of the form $x_i$ and $y_S$ (by the definition of ``below'') and for those $v$ we also find that $f(u)\not\subseteq f(v)$ since $[b+1,b+\ell]\not\subseteq f(v)$ whereas $[b+1,b+\ell]\subseteq f(u)$. 
When $v=z_{j'}$ for $j'\in [b+\ell+1,n-a+\ell]$, the proof is analogous to before: if $z_j\not\cleq_P z_{j'}$, then $j'\in f(z_{j})\setminus f(z_{j'})$, whereas transitivity of $\cleq_P$ ensures that $f(z_j)\subseteq f(z_{j'})$ when $z_j\cleq_P z_{j'}$.

Since we proved (\ref{eq:poset_contains}) in all cases for $u$ and $v$, this shows that $(2^{[n-a+\ell]},\subseteq )$ contains the poset $P$ via the embedding given by $f$.
\end{proof}

Combining the two constructions now yields our result.
\begin{proof}[Proof of Theorem~\ref{thm:uni_poset_2/3_hypercube}]
   We may assume $n$ is sufficiently large, since the Boolean lattice itself is a universal poset and we may choose the constant $C$.
   For a fixed $a \in [n]$ in Lemma~\ref{lem:hypercube_chain} and Lemma~\ref{lem:hypercube_antichain} the union of the two set systems resulting from the lemmas yields a set system $V\subseteq 2^{[n]}$ for which $(V,\subseteq)$ is universal for the $n$-element posets. We minimise the size of $V$ by setting $a=\lceil n/3\rceil$ and obtain
\[
|V|\leq p(n) (3+1)^{n/3}+2^{2n/3+(1+o(1))\log_2 n}=2^{2n/3+O(\sqrt{n})}, 
\] 
using that the number of partitions of $n$ is $p(n)=2^{O(\sqrt{n})}$ by Theorem~\ref{thm:partition}.
\end{proof}

\section{Conclusion}
We showed that one can find a universal poset for posets on $n$ elements of order $2^{(1+o(1))2n/3}$ inside the Boolean lattice. In particular, this is now also the best upper bound for the following problem, with the best lower bound of the form $2^{n/4+o(n)}$. The maximal size of an antichain in a poset $P$ is well-studied and known as the \emph{width} of $P$. We reduced \cref{pb:uni_min_poset} to the study of posets with width $w \in [n/17, 3n/4]$. Indeed, \cref{lem:hypercube_chain} implies that there exists a universal poset $U$ of size $2^{n/4 + o(n)}$ for the family of posets of width at most $xn$ where $x\leq 1/17$ is the solution to $\log_2(1/x+1)x=1/4$
and \cref{lem:hypercube_antichain} gives the existence of such a poset for the family of posets of width at least $3n/4$.

\begin{problem}
\label{prob:3}
    What is the minimum order of a subposet of the Boolean lattice $(2^{[n]},\subseteq)$ that contains all posets on $n$ elements?
\end{problem}

Our initial construction of a universal poset on $2^{2n/3+o(n)}$ elements for posets on $n$ elements was via a labelling scheme, similar to the one mentioned in the introduction and used by  \cite{DulebaGawrychowskiJanczewski,bonamy2021optimal}. The subtlety is to ensure that the labelling created also preserves the transitivity (see the PhD thesis of the first author~\cite{paulthesis}). It could be that such a proof method will be able to construct a smaller universal poset that is not contained in the Boolean lattice, but we were unable to improve further the constant $2/3$ in the exponent.

Recall from Proposition~\ref{prop:hypercube_univ} that the ``folklore'' embedding of an $n$-vertex poset $(P,\cleq_P)$ into $U=(2^{[n]},\subseteq )$ works by fixing some order $v_1,\dots,v_n$ on the elements of $P$ and then embedding $v_j$ to $S_j=\{i\in [n]\mid v_i\cleq_P v_j\}$. We showed that there is a much smaller subposet $U'$ of $U$ such that for all posets without a large antichain, there is a way to choose the order on the elements such that the poset is guaranteed to embed into $U'$. This leaves the question whether we could have done a similar thing for all posets (possibly with a large antichain). The following proposition shows that this is not possible.

First, we give some intuition. Let $\mathcal{P}$ be the family of $n$-element posets $P$ for which all but $\sqrt{n}$ of the elements are minimal. Let $\mathcal{S} \subseteq 2^{[n]}$ be a universal subposet of the hypercube for the family $\mathcal{P}$. If we restrict the representation of each $P \in \mathcal{P}$ to be given via a labeling $v_1,\ldots,v_n$ of the elements of $P$ such that $v_j$ is represented by $\{i \in [n]\mid v_i \cleq_P v_j\} \in \mathcal{S}$ for all $j \in [n]$, as mentioned above, then any minimal element will be represented by a singleton set in $\mathcal{S}$. For the remaining $\sqrt{n}$ sets $S$ representing a non-minimal element of $P$, this determines whether $i\in S$ for all but $\sqrt{n}$ of the $i\in [n]$. This forces us to use many different sets.
In \cref{lem:hypercube_antichain}, we manage to circumvent this issue and construct a significantly smaller universal subposet of the hypercube by reducing the number of elements needed to represent the antichain of minimal elements: instead of using singletons, we use subsets of size $\ell/2$ from a fixed set of size $\ell$, where $\ell = (1 + o(1))\log_2(n)$.

\begin{proposition}
    Let $\mathcal{S}\subseteq 2^{[n]}$. Suppose that for each poset $(P,\cleq_P)$ there is an order $v_1,\dots,v_n$ on $P$ such that for each $j\in [n]$,
\[
\{i\in [n]\mid v_i\cleq_Pv_j\}\in \mathcal{S}.
\]
Then $|\mathcal{S}|\geq 2^{(1-o(1))n}$.
\end{proposition}
\begin{proof}
For simplicity, we will assume $m=n-\sqrt{n}$ is an even integer. We next define a family $\mathcal{P}\subseteq \binom{2^{[n]}}{n}$ of set families. A set family $P\subseteq 2^{[n]}$ is a member of $\mathcal{P}$ if it contains exactly $n$ subsets, $m$ of which are $\{1\},\dots,\{m\}$ and the remaining $n-m=\sqrt{n}$ are of the form $\{m+i\}\cup T_i$ for $i\in [\sqrt{n}]$ and $T_i\in \binom{[m]}{m/2}$. We call the sets $\{1\},\dots,\{m\}$ the \textit{bottom elements} and the remaining sets the \textit{top elements} of $P$.

Each element $P\in \mathcal{P}$ gives a (labelled) poset $(P,\subseteq )$ on $n$ elements. 
Note that
\[
|\mathcal{P}|\geq 
\binom{m}{m/2}^{\sqrt{n}}=
2^{(1-o(1))n\sqrt{n}}.
\]
Let $\mathcal{S}\subseteq 2^{[n]}$ be given such that for each $P\in \mathcal{P}$, there is an order $v_1,\dots,v_n$ on the elements of $P$ such that $v_j'=\{i\in [n]\mid v_i\cleq_P v_j\}\in \mathcal{S}$ for all $j\in [n]$. 
In particular, for each $P\in \mathcal{P}$, we may record the $\sqrt{n}$ images $v_{i_1}',\dots,v_{i_{\sqrt{n}}}'$ of the top elements $v_{i_1},\dots,v_{i_{\sqrt{n}}}$ of $P$. Using a double counting argument, we will show that 
\[
\binom{|\cS|}{\sqrt{n}} \geq |\mathcal{P}|/n!=2^{(1-o(1))n\sqrt{n}}
\]
which will imply that $|\cS|\geq 2^{(1-o(1))n}$, as desired.

To prove the inequality above, we show that for every choice of $\sqrt{n}$ elements $S_1,\dots,S_{\sqrt{n}}\in \mathcal{S}$, there are at most $n!$ choices of $P\in \mathcal{P}$ for which $S_1,\dots ,S_{\sqrt{n}}$ are the top elements of $P$.

Let $S_1,\dots,S_{\sqrt{n}}\in \mathcal{S}$ and let $\sigma:[n]\to[n]$ be a bijection. Suppose that $P\in \mathcal{P}$ is such that the fixed order $v_1,\dots,v_n$ on the elements of $P$ (provided by the assumed `embedding algorithm') has the property that $\{i\}=v_{\sigma(i)}$ for $i\in [m]$, $i \in v_{\sigma(i)}$ for $i \in \{m+1, \ldots, n\}$, and 
$(S_1,\dots,S_{\sqrt{n}})=(v_{\sigma(m+1)}',\dots,v_{\sigma(n)}')$. 
We prove that there is at most one such~$P$. The bottom elements of $P$ are always the same and its top elements are of the form  $v_{\sigma(m + i)}=\{m+i\}\cup T_i$ for $i\in [\sqrt{n}]$ and $T_i\in \binom{[m]}{m/2}$. Noting that the top elements are incomparable, for $i\in [\sqrt{n}]$ we have \[
S_i=v'_{\sigma(m+i)}=\{j\in [n]:v_j\subseteq v_{\sigma(m + i)}\}=\{\sigma(m + i)\}\cup\{\sigma(j) : j \in T_i\},
\]
where the last part follows from $v_{\sigma(j)} = \{j\}$ for $j \in [m]$. In particular, $T_i$ is determined uniquely by $S_i$ and $\sigma$.
\end{proof}

\paragraph{Acknowledgement}
We would like to thank Bal\'azs Patk\'os for pointing out a mistake in a correspondence we made between Problem~\ref{pb:uni_min_poset} and Problem~\ref{pb:uni_min_comparability} in an earlier version of this paper (which worked for adjacency labelling in comparability graphs and comparability labelling but did not preserve the poset structure). We also thank the anonymous referees for useful suggestions for improvement.

    \newcommand{\etalchar}[1]{$^{#1}$}

\end{document}